\documentclass{amsart}%[a4paper,12pt,reqno]{amsart}
\usepackage[english]{babel}
\usepackage[T1]{fontenc}
\usepackage{amsmath}
\usepackage{amsfonts}
\usepackage{amssymb}
\usepackage{latexsym}
\usepackage{mathrsfs}
\usepackage{delarray}
\usepackage{amssymb,amsmath,amsfonts,amsthm,mathrsfs}
 \newtheorem{thm}{Theorem}[section]

 \newtheorem{prop}[thm]{Proposition}
 \theoremstyle{definition}
 \newtheorem{defn}[thm]{Definition}
 \newtheorem{obs}[thm]{Observation}
 \theoremstyle{remark}
 \newtheorem{rem}[thm]{Remark}
 
 \numberwithin{equation}{section}

\newcommand{\bi}{\begin{itemize}}
\newcommand{\ei}{\end{itemize}}
\newcommand{\be}{\begin{enumerate}}
\newcommand{\ee}{\end{enumerate}}
\newcommand{\beq}{\begin{equation}}
\newcommand{\eq}{\end{equation}}

\begin{document}
%-------------------------------------------------------------------------
% editorial commands: to be inserted by the editorial office
%
%\firstpage{1}
%\volume{228}
%\DOI{003-0001}
%\seriesextraline{This is the Concrete Title of this Book\br H.E. R and S.T.C. W, Eds.}
%\firstpage{1}
\title[A note on $L^p_{w}(\nu,X,Y)$ spaces]
 {A note on $L^p_{w}(\nu,X,Y)$ spaces of vector-valued functions with respect to vector measures}
%----------Author 1
\author[Liliana Posada Vera]{Liliana Posada Vera}

\address{
Departamento de Matem\'aticas\\
Universidad del Valle\\
Calle 13 100-00. Cali, Departamento del Valle del Cauca, Colombia.
}

\email{liliana.posada@correounivalle.edu.co}

\thanks{The author was supported by Universidad del Valle.}
%----------Author 2
%\author[dd]{dd}
%\address{
%Menadostraat 22\\
%3532SM, Utrecht,\\
%The Netherlands.}
%\email{hhh}


\date{\today}
\dedicatory{Dedicated to the memory of Professor Guillermo Restrepo Sierra}
%%% ----------------------------------------------------------------------
\begin{abstract}
 In this work we introduce the $L^p_{w}(\nu; X; Y)$ spaces for the case where $\nu$ is a vector measure and the functions are vector-valued. We establish fundamental properties for such spaces.
\end{abstract}
%%% ----------------------------------------------------------------------
\maketitle
%%% ----------------------------------------------------------------------
%\tableofcontents

\section{Introduction}
$L^p$ spaces play a crucial role in modern mathematics. Their properties, among other things, have originated fundamental results in several fields such as differential equations,  measure theory, and functional analysis. In this paper, we will pursue two aims. Firstly, to define $L^p_ {w}(\nu,X,Y)$ spaces for $X$-valued functions, where $X$ and $Y$ are Banach spaces, and $\nu$ is a vector measure valued on $Y$. Secondly, to establish some properties as separability, density of simple functions, and the H\"older inequality for these spaces.\\ The pioneering works in geometry and the theory of the  Banach spaces in infinite dimension have their theoretical origins in vector measures. An example of this is found in the theorem of Orlicz-Pettis, which can be considered as a link between the theory of vector measures and properties of the Banach spaces, because, it relates the convergence of series in Banach spaces with $\sigma$- additive vector measures.\\
A paramount question in the theory of vector measures is to determine the existence of Radon-Nikodym derivatives of a vector measure with respect to another.
This condition has crucial consequences in the study of weak compactness in spaces of functions such as the Lebesgue-Bochner space $L^1(\mu, X)$, where $(\Omega, \mathcal{A},\mu)$ is a finite measure space.
Furthermore, it is essential in the Lebesgue-Bochner space duality theorem $L^p(\mu,X)$, where its respective dual is $L^q(\mu,X)$ if and only if $X'$ has the property of Radon-Nikodym for $p^{-1} + q^ {-1} = 1$.
For further applications of vector measures, such as in game theory, control theory, stochastic differential equations, and statistics, just to mention a few, see for example \cite{EO} and \cite{GL}.

%The extension of classical concepts in the field of measure theory begins with the notion of a vector measure. One of the main concepts in the development of functional analysis is one of a vector measure.
%As we know the space $L^p$ is one of the main spaces of functions in analysis and its versions Sobolev on an open of $\mathbb{R}^n$ or a smooth variety. It is well known from classical literature that Sobolev spaces are important in the search for solutions to solve partial differential equations.

The introduction of integrable scalar or vector-valued functions with respect to vector measures leads to define interesting $L^p$ spaces, which might be used in other areas of mathematics. For example, in \cite{OB} and \cite{DM}, some properties of this sort of spaces for scalar functions and vector measures were extended yielding to important results in the harmonic analysis theory.
This type of applications motivates us to explore more about these spaces and their connections with other areas.\\
For this purpose we will focus on studying $L^p$ spaces for strongly measurable functions but weakly integrable with respect to a vector measure and the injective tensor product. In particular, we will use the definition of a $\otimes$-integrable function given in \cite{SG} and the characterization of the 
$\otimes$-integrable functions of order $p$, for 
$ 1 \leq p < \infty $.\\
Let $(\Omega, \mathcal{A})$ be a measurable space. We will say that a vector-valued function is $\otimes$ -integrable of order $p$ if and only if $\|f\|^p$ is $\nu$-integrable, i.e., $\|f\|^p$ is 
$|y'\circ\nu|$-integrable for all $ y' 
\in Y'$ and  for each $E\in\mathcal{A }$ there exists an element of $Y$ denoted by 
$\int_{E}\|f\|^p d\nu,$ such that $$y'\left(\int_{E}\|f\|^pd\nu\right) = \int_{E}\|f\|^pd(y'\circ 
\nu), \hspace{0.2cm} \textnormal{for each} \hspace{0.2cm} y'\in Y'.$$  We will base our construction of the $L^p_{w}(\nu,X,Y)$ spaces on $\otimes$-integrable functions of order $p$, for $1\leq p<\infty$,  that only satisfy the first condition. We will determine their properties, and by means of an extension of the norm introduced in \cite{CB2}, we will prove that $L^p_{w}(\nu,X,Y)$ is a Banach space. We will proved that the spaces $L^p_{w}(\nu,X,Y) $ are separable showing that the set of simple functions is dense in $L^p_{w}(\nu,X,Y)$, and that when the $\sigma$-algebra is countably generated  and $X$ is a separable Banach space, $L^p_{w}(\nu,X,Y)$ is separable. Additionally we will show that $L^p(\nu,X,Y)$ is a closed subspace of $L^p_{w}(\nu,X,Y)$ which is, a generalization of the proof by Stenfansson in \cite{SG1}. We point out the difficulty of dealing with these spaces. Since their duals might not coincide with $L_w^q$ versions. Even in the case of $X=\mathbb{R}$ the pathology remains. At the end of this note we will discuss this issue in more detail.

\vspace{0.2cm}

\section{Preliminares}
To start with, we will use the concept of $\nu$-integrability that was introduced by Lewis in \cite{L} for the case scalar valued-functions and vector measures. 
\vspace{0.2cm}

\begin{defn} \label{def1: definitionofintegrabilitylewis}
Let $(\Omega, \mathcal{A})$ be a measurable space, $f$ be a scalar valued  measurable function, $Y$ be a Banach space, $y'\in Y'$ with $Y'$ the dual of $Y$, and $\nu: \mathcal{A} \to Y$ a  $\sigma$-additive vector measure. A  $\nu$ -measurable function $f$ is \textit{$\nu$ -integrable} if
\be
\item $f$ is $y'\circ \nu$ -integrable for each $y'\in Y'$, i.e., $$\int_{\Omega} 
|f|d|y'\circ\nu|< \infty,$$
\item for every $E\in\mathcal{A}$ there exists an element of $Y$ denoted by $\int_{E} fd\nu,$ such that $$y'\left(\int_{E}fd\nu \right) = \int_{E}fd(y'\circ\nu),$$ for each $y'\in Y'$.
\ee
\end{defn}
\vspace{0.2cm}

Using this definition, on $1975$ Kluv\'anek and Knowles defined in \cite{KK} the  $L^1(\nu)$ space as the space of the $\nu$ -integrable functions that satisfy both conditions. Years later, G. Curbera in \cite{Cb} and \cite{GCb} established additional properties of this space and Sanchez in \cite{Sn} extended these studies para $1<p<\infty$.
\vspace{0.2cm}

Next, consider the following definition:
\vspace{0.2cm}

\begin{defn}
Let $1 < p <\infty$, and $\nu$ be a vector measure $\sigma$ -additive. We will say that a scalar measurable function  is \textit{$p$-integrable} with respect to $\nu$ if 
$|f|^p$ is $\nu$ -integrable.
\end{defn}

S\'anchez showed that with the norm $$\|f \|_{L^{p}(\nu)} = \sup\limits_ {\|y'\|\leq 1} \left \{\left(\int_{\Omega} |f|^{p}d|y'\circ\nu|\right)^{1/p} \right \}, $$ $L^p(\nu)$ is a Banach space and the set of simple functions is dense in $ L^p(\nu)$.
\vspace {0.2cm}

Morever, in \cite{Sn}, S\'anchez illustrated by means of on example an interes-ting situation. If $p$ and $q$ are such that $p^ {-1} + q ^ {-1} = 1 $, then the spaces $L^{q}(\nu)$ and $(L^{p}(\nu))'$ are different.\\

In \cite{SG1}, Stef\'ansson defined that a scalar valued measurable function  in the usual sense is \textit{weakly integrable},  if for all $ y'\in Y'$, $f$ is $y'\circ \nu$ -integrable and denotes $L^1_{w}(\nu)$ as the set of weakly integrable functions. Additionally, it showed that $L^ 1_ {w}(\nu) $ is a Banach space with the norm $$\|f\|_{L^{1}_ {w}(\nu)} = \sup 
\limits_{\|y'\| \leq 1} \left \{\int_ {
	\Omega} |f|d|y'\circ \nu| \right\}, $$ which  contains $ L^1 (\nu) $ as a closed subspace.\\

In \cite{DM}, for $ 1 <p <\infty $ the space $ L^{p}_ w (\nu) $ is defined.
\vspace{0.2cm}

\begin{defn}
	Let $ 1 <p <\infty $. The space $L^{p}_w (\nu) $ is defined as the space ($\|\nu\|$-equivalence classes of) of the measurable functions such that $|f|^p \in L^1_ {w} (\nu)$ .
\end{defn}
It can be shown that with the norm
$$ \| f \|_ {L^{p}_w(\nu)} = \sup\limits _ {\|y'\|\leq 1} \left \{\left (\int_ {
	\Omega}|f|^{p}d|y'\circ\nu|\right)^ {1/p} \right\}, $$ $ L^p_{w}(\nu) $ is a Banach space. Additionally, the following relationships between these spaces hold.
$$L^{p}(\nu) \subset L^{p}_w(\nu) 
\subset L^1_{w}(\nu).$$
\begin{rem}
	Thanks to Rybakov's theorem, we know that for every vector measure $\nu$ there exists a real measure of control $\mu$ that makes the vector measure 
	$\nu$ to be $\mu$ - continuous. Further, in \cite{DM} it is showed that for $p = 
	\infty$, the space $L^{\infty}(\nu)$ of the  measurable functions that are bounded $\nu$ -c.t.p, coincide with the space $L^{\infty}(\mu)$. That is the reason that do not consider this case.
\end{rem}
\vspace{0.2cm}

Following Stenf\'ansson in \cite{SG}, to extend  the notion of these spaces to  vector-valued and $\nu$ -measurable functions. We will say that a  vector- valued and $\nu$-measurable function $f$ is \textit{$\otimes$ -integrable}, if there exists a sequence of simple functions $(f_ {n})_ {n\in\mathbb {N}}$ of $ \Omega $ in $ X $ such that
$$ \lim\limits_{n \to \infty} \sup 
\limits_ {\|y'\| \leq 1} \left \{\int_ 
{\Omega} \| f-f_n \|d|y' 
\circ \nu| \right\} = 0. $$
From this definition, we make the following observations.
\vspace{0.2cm}

	\be
	\item It can be shown from this definition that $\left (\int_ {
		\Omega} f_{n} \otimes d \nu 
	\right)_ {n \in \mathbb{N}} $ is a Cauchy sequence on $X \widehat{
		\otimes}_{\epsilon} Y $.
	\item This integral is defined for each $ E\in\mathcal{A} $ and the $ \lim 
	\limits_ {n \to \infty} \int_{E} f_n 
	\otimes d \nu $ exists and is unique since $ X\widehat{\otimes}_ {\epsilon} 
	Y $ is a Banach space.
	\item If $f$ is $\otimes$ -integrable then $\int_{E}f\otimes d\nu $ is the vector such that $$ \int_{E} f 
	\otimes d\nu = \lim \limits_ {n \to 
		\infty} \int_{E} f_n \otimes d 
	\nu. $$
	\item It can be shown that this 
	limit does not depend on of the choice of $(f_n)_{n \in \mathbb{N}} $.
	\ee

\vspace{0.2cm}

Let $f$ be a $\nu$ -measurable function. In \cite{SG}, Stenf\'ansson  defines the space $L^1(\nu,X,Y)$ as the vector space of all ($\|\nu\|$-equivalence classes of) functions $\otimes$ -integrables equipped with the norm
$$\| f \|_{L^1(\nu,X,Y)} = \sup\limits_{\ |y'\| \leq 1} \left \{\int_{\Omega} \| f \|d|y'\circ \nu| \right\},$$ and shows that is a Banach space.
\vspace{0.2cm}

Furthermore, in \cite{CB1} additional properties of $L^1 (\nu,X,Y)$ are shown as the fact of being a Banach lattice, a separable space, and the density of the set of simple functions in $L^1 (\nu,X,Y)$.\\
 In \cite{CB2}, the spaces $ L^p(\nu,X,Y) $ for $1 <p <\infty $ are defined as follows.
\vspace{0.2cm}

\begin{defn}
	A  $ \nu $ -measurable function $ f: \Omega \to X $  is called \textit{$\otimes$ -integrable of order $p$}, if there exists a sequence of simple functions $ (f_n) _ {n \in \mathbb{N}}$ of $\Omega$ in $ X $ such that $$ \lim \limits_{n \to 
		\infty} \sup \limits_ {\|y'\| 
		\leq 1} \left \{\left(\int_{
		\Omega} \|f (w)-f_n (w) \|^p| d|y'\circ \nu | \right)^{1 / p} 
	\right\} = 0. $$
\end{defn}
\vspace{0.2cm}

As a consequence of this definition, it can be stated that if 
$f$ is $\otimes$ -integrable of order $p$, then the expression $$ 
\|f \|_{L^p(\nu,X,Y)} = \sup 
\limits_ {\|y'\| \leq 1} \left 
\{\left(\int_{\Omega}\| f \|^ pd| y' \circ \nu| \right)^{1/p} 
\right\},$$  is a norm, and hence $L^p(
\nu,X,Y)$ is a Banach space.
\vspace{0.2cm}

Additionally, in \cite{CB2}, Chakraborty and Basu, considered some properties similar to those given by Stenf\'ansson in 
\cite{SG}. One of them is that $f$ is $\otimes$ -integrable of order $p$ if and  only if 
$ \| f \|^p$ is $\nu$ -integrable, i.e., $\|f\|^ p \in L^1(\nu) $. (See Theorem 
$1$ in \cite{CB2}, page $90$). Remarkably, this property allowed to show the existence of the dominated convergence theorem of order 
$p$.
\vspace{0.2cm}

In addition to the above, and as a generalization of S\'anchez work in \cite{Sn}, Chakraborty and Basu  showed  that $L^p (\nu,X,Y)$ is a Banach lattice, $L^p(\nu,X,Y) \subset L^1(\nu,X,Y) $ for $1<p<\infty $, and exhibits a dual of the spaces $ L^p(\nu,X,Y) $ .

Furthermore, in \cite{CB2}, they generalized S\'anchez's spaces in the case of real-valued  weakly integrable functions, considering weakly measurable functions defining the w- $L^p (\nu, X, Y)$  spaces and establishing several properties.

	\section{Results}
	
	In this section we will introduce the $L^p_{w}(\nu,X,Y)$ spaces for functions $\nu$ -measurable and integrable with respect to $ |y'\circ\nu|$ for all $ y'\in Y'$. Additionally we will state properties of separability and density of simple functions.
    We remark that we are based on the Theorem $1$ in 
	\cite{CB2} which states that $f$ is $\otimes$ -integrable of order $p$ if and only if $\|f \|\in L^p(\nu)$.
	\vspace{0.2cm}
	
	\begin{defn}
	Let $1 \leq p <\infty $. We define 
	$L^p_{w}(\nu,X,Y)$ as the space of \textit{the functions $\nu$ -measurable ($\|\nu 
	\|$ equivalence classes) such that 
	$\|f \|^ p$ is $|y'\circ\nu|$ -integrable}, that is,
	$$ \int_ {\Omega} \|f \|^pd|y'\circ 
	\nu|<\infty, $$ for each $y' \in Y'$.
	\end{defn}

If the expression $$\sup\limits_{\|y'\| 
	\leq 1} \left \{\left(\int_{\Omega} 
\|f\|^pd|y'\circ\nu| \right)^{1/p}\right \} <\infty, $$ then it can be shown that it is a norm that we will denote by $ \|f \|_ {L^p_{w}(\nu,X,Y)}$. In addition, under this norm it can be proved that  $L^p_{w}(\nu,X,Y)$ is a Banach space, and the proof of this fact is similar  to that given  in \cite{SG} by Stenf\'ansson for the case $p=1$.\\

%To do this we will use the definition of $\nu$ -integrability given by Bartle in \cite{BDS} and is equivalent to that given by Lewis in \cite{L} for Banach spaces and that we will apply for the real measurable function $\|f\|$.

%\begin{defn}
%	We will say that a measurable function of real value $f$ is $\nu$ -integrable if there is a succession of simple functions $(f_n)_{n\in\mathbb{N}}$ such that 
%	\begin{enumerate}
%		\item $f_n(w)\to f(w)$ $\nu$-c.t.p. \item The sequence $\left(\int_{E}f_nd\nu\right)_{n\in\mathbb{N}}$ converge in the norm of $Y$ for each $E\in\mathcal{A}$.
%	\end{enumerate}
%\end{defn}  
\vspace{0.2cm}
\begin{obs}
	Let be $1<p<\infty$.
Note that if $X=\mathbb{R}$ then $L^p_{w}(\nu,X,Y)=L^p_{w}(\nu)$.
If $Y=\mathbb{R}$ then $L^p_{w}(\nu,X,Y)=L^p(\nu,X)$ where $L^p(\nu,X)$ corresponds to Lebesgue-Bochner spaces.
\end{obs}
\vspace{0.2cm}
We will show that $L^p(\nu,X,Y)$ is a closed subspace of $L^p_{w}(\nu,X,Y)$. 
\vspace{0.2cm}

\begin{thm}
	$L^p(\nu,X,Y)$ is a closed subspace of $L^p_{w}(\nu,X,Y)$.
\end{thm}
\begin{proof}
	Let be $f\in L^p_{w}(\nu,X,Y)$ and $(f_m)_{m\in 
		\mathbb{N}}$ be a sequence of $\nu$ -measurable functions in $L^p(\nu,X,Y)$ that converges to $f$ in $L^p_{w}(\nu,X,Y)$. Then $(\|f_m\|^p)_{m\in \mathbb{N}}$ is a sequence of measurable functions in $L^1(\nu)$. If we define
	$$\mu_m(E)=\int_{E}\|f_m\|^pd\nu, \hspace{0.2cm} \textnormal{and} \hspace{0.2cm} \mu(E)=\int_{E}\|f\|^pd\nu,$$
	then $\mu_m$ and $\mu$ are vector measures $\sigma$-additive and therefore
	$$\|\mu_m(E)-\mu(E)\|\leq \|\mu_m-\mu\|(E)\to 0,$$ as
\begin{align*} \|\mu_m-\mu\|(E)
& =\sup\limits_{\|y'\|\leq 1}\left\{\int_{E}|\|f_m\|^p-\|f\|^p|d|y'\circ\nu|\right\}\\
& \leq \sup\limits_{\|y'\|\leq 1}\left\{\int_{E}\|f_m-f\|^pd|y'\circ\nu|\right\}\\
& = \|f_m-f\|_{L^p_{w}(\nu,X,Y)}\to 0
\end{align*}

when $m\to \infty$ and holds for all $E\in \mathcal{A}$. And hence $\|f\|^p \in L^1(\nu)$. In conclusion $f\in L^p(\nu,X,Y)$.
	% $$\lim\limits_{n\to \infty}\sup\limits_{\|y'\|\leq 1}\int_{\Omega}\|f_{mn}-f_m\|d|y'\circ\nu|=0.$$
	%By other side $$\int_{\Omega}\|f_{mn}-f\|d|y'\circ \nu|\leq \int_{\Omega}\|f_{mn}-f_m\|d|y'\circ\nu|+\int_{\Omega}\|f_{m}-f\|d|y'\circ\nu|,$$ taking the supreme for $\|y'\|\leq 1$, we have
	%$$\sup\limits_{\|y'\|\leq 1}\int_{\Omega}\|f_{mn}-f\|d|y'\circ \nu|\leq \sup\limits_{\|y'\|\leq 1}\int_{\Omega}\|f_{mn}-f_m\|d|y'\circ\nu|+\|f_{m}-f\|_{L^1_{w}(\nu,X,Y)}.$$ 
	%For fixed $m$,
	%$$\lim\limits_{n\to \infty}\sup\limits_{\|y'\|\leq 1}\int_{\Omega}\|f_{mn}-f\|d|y'\circ \nu|\leq\|f_{m}-f\|_{L^1_{w}(\nu,X,Y)},$$ and now if $m\to \infty$, then
	%$$\lim\limits_{n\to \infty}\sup\limits_{\|y'\|\leq 1}\int_{\Omega}\|f_{mn}-f\|d|y'\circ \nu|=0.$$
\end{proof}

Of the previously defined, we can also show that

$$L^p(\nu,X,Y)\subset L^p_{w}(\nu,X,Y)\subset L^1_{w}(\nu,X,Y)$$ and $$L^p(\nu,X,Y)\subset L^1(\nu,X,Y)\subset L^1_{w}(\nu,X,Y).$$
\vspace{0.2cm}
	
We are going to state properties of density and separability for these spaces.

	\begin{thm}
Let	$1\leq p<\infty$. The set of simple functions is dense in $L^p_{w}(\nu,X,Y)$.
	\end{thm}
	\begin{proof}
	Let $f\in L^{p}_w(\nu,X,Y)$. As a consequence of the Pettis measurability theorem for vector measures, there exists a sequence of functions $(f_n)_{n \in 
		\mathbb{N}}$  $\nu$-measurables that  only assume countably many values such that $\|f_n-f \| \leq 1/n $, $ \nu$ -c.t.p. Hence  $$ \| f_n \|^p \leq 2^p\left (\| f \|^p + \frac{1}{n^p} 
	\right),$$ and  $$\sup\limits_{\|y'\|\leq 1}\int_ {\Omega} \| f_n \|^pd|y'\circ \nu| <\infty,$$ that means
	 that $(f_n)_{n\in\mathbb{N}}\subset L^p_{w}(\nu,X,Y)$.
	We can show that $$\lim\limits_{\|\nu\|(E)\to 0}\sup\limits_{\|y'\|\leq 1}\int_{\Omega}\|f_n\chi_E\|^pd|y'\circ\nu|= 0$$ for each $E\in\mathcal{A}$. Indeed, writing $$f_n=\sum\limits_{m=1}^{\infty}x_{nm}\chi_{E_{nm}},$$ where $E_{nk}\cap E_{nj}= \emptyset$ if $k\neq j$ and $x_{nm}\in X$, we obtain for each
$y'\in Y'$ the next:
\begin{align*}
\int_{\Omega}\|f_n\chi_E\|^pd|y'\circ\nu| & =\int_{\Omega}\sum\limits_{m=1}^\infty\|x_{nm}\|^p \chi_{E_{nm}\cap E}d|y'\circ \nu| \\
& = \sum\limits_{m=1}^\infty\int_{\Omega}\|x_{nm}\|^p \chi_{E_{nm}\cap E}d|y'\circ\nu|\\
& = \sum\limits_{m=1}^\infty\|x_{nm}\|^p|y'\circ\nu|\left(E_{nm}\cap E\right)\\
& \leq \sum\limits_{m=1}^\infty\|x_{nm}\|^p|y'\circ\nu|\left(E\right)
\end{align*}
Taking the supremo on $\|y'\|\leq 1$, we obtain
\begin{align*}
\sup\limits_{\|y'\|\leq 1}\int_{\Omega}\|f_n\|^pd|y'\circ\nu| & \leq \sup\limits_{\|y'\|\leq 1}\sum\limits_{m=1}^\infty\|x_{nm}\|^p|y'\circ\nu|\left(E\right)\\
& = \sup\limits_{\|y'\|\leq 1}|y'\circ\nu|\left( E\right)\sum\limits_{m=1}^\infty\|x_{nm}\|^p \\
& = \|\nu\|\left(E\right)\sum\limits_{m=1}^\infty\|x_{nm}\|^p \to 0
\end{align*}
when $\|\nu\|(E)\to 0$.

 For each $ n $, we can choose a $ p_n $ large enough so that $$ \sup\limits_{\|y'\|\leq 1}\int_{\bigcup \limits_ {m= p_ {n} +1}^ \infty E_ {nm}} \| f_n \|^pd |
y'\circ \nu | \leq \frac {\| \nu \| (
	\Omega)}{n}. $$ If $ \phi_n = \sum 
\limits_{m = 1}^{p_n} x_ {nm} \chi_ {E_ {nm}} $, for all $y'\in Y'$ we have 
\begin{align*}
\int _ {\Omega} \| f- \phi_n \|^p d|y'
\circ \nu| & \leq 2^p \left \{\int_ {
	\Omega} \|f-f_n \|^pd|y'\circ \nu | + \int_ {\Omega} \|f_n- \phi_n \|^pd|y'
\circ \nu | \right \} \\
& \leq 2^p \left [\frac {\| \nu \| (
	\Omega)} {n^p} + \int _ {\Omega} 
\left \| \sum \limits_{m = p_n + 1}^ 
\infty x_ {nm} \chi_ {E_ {nm}} \right \| ^ pd |y'\circ \nu|\right]
\end{align*}
Taking the supremum over all $ \| y'\| 
\leq 1 $, we get
\begin{align*}
\sup\limits_{\|y'\| \leq 1} \int_ {
\Omega} \| f- \phi_n \|^pd |y'\circ \nu| & \leq 2^p \left [\frac{\|\nu\|(\Omega)} {n^p} + \int_{\bigcup\limits_{p_{n} +1} ^ \infty E_{nm}} \| f_n \|^pd |y '\circ \nu | \right] \\
& \leq 2^p \left [\frac {\| \nu \| (
	\Omega)} {n^p}+ \frac{\| \nu \| (
	\Omega)}{n} \right]
\end{align*}
Then $$\lim\limits_{n\to\infty} \| f- 
\phi_n \|_ {L^p_{w}(\nu,X,Y)} = 0,$$ which completes the proof.
\end{proof}

Next, we will show that the $L^p_{w}(\nu, X, Y)$ spaces are separable.

\begin{thm}
Let $ (\Omega, \mathcal{A})$ be a measurable space and be $p$ a real such that $1 \leq p <\infty $. If $X$ is a separable Banach space and the $\sigma$- algebra is countably generated, then the space $ L^p_{w}(\nu,X,Y) $ is separable.
\end{thm}
\begin{proof}
	Let $D$ be a dense and countable subset of $X$. Since $\mathcal{A}$ is countably generated then there exits a countable set $
	\mathcal{F} \subset \mathcal{A}$.
	We consider the collection of all finite sums $\sum \limits_{j=1}^ n q_j \chi_{D_j} $ where $q_j \in D$ and $D_j \in \mathcal{F}$. This collection is countable and is contained in the space $L^p_{w}(\nu,X,Y) $. In what follows, we will show that this collection is a dense subset.\\
	\vspace{0.1cm}
	
Let $f \in L^p_{w}(\nu,X,Y)$ and 
	$\epsilon > 0$. Since the set of simple functions is dense in $L^p_{w}(
	\nu,X,Y) $. There exists a simple function $g$ such that $$ 
	\|f-g \|_{L^p_{w}(\nu,X,Y)} \leq 
		\epsilon.$$ From now on, we assume that $ g $ is written as $ \sum 
	\limits_{j=1}^n\alpha_j \chi_{A_j}$ where each $\alpha_{j}\in X$ and every $A_j \in \mathcal{A}$. Our purpose here is to show that for this function $ g $, there exist $ q_j \in D $ and
	 $D_j\in \mathcal{F}$ such that
$$\left\|g-\sum\limits_{j=1}^nq_j\chi_{D_j}\right\|_{L^p_{w}(\nu,X,Y)}\leq  \epsilon.$$
We begin by noting that since $X$ is separable for each $\alpha_j$ we can choose a $q_j$ such that
$$\|\alpha_j-q_j\|\leq \frac{\epsilon}{(2n\|\nu\|(\Omega))^{1/p}},$$ for $1\leq j \leq n$. For $ 
\mathcal{A}$ countably generated we have 
$$\sup\limits_{\|y'\|\leq 1}\int_{\Omega}|\chi_{A_j}-\chi_{D_j}|d |y'\circ \nu|\leq \frac{\epsilon^p}{2n\beta^p},$$ where $\beta=\sup\limits_{1\leq j\leq n}\|\alpha_j\|$.
Then

\begin{align*}
\left\|\sum\limits_{j=1}^n \alpha_j \chi_{A_j}-\sum\limits_{j=1}^n q_j \chi_{D_j}\right\|_{L^p_{w}(\nu,X,Y)}^p & = \sup\limits_{\|y'\|\leq 1}\int_{\Omega}\left\|\sum\limits_{j=1}^n \alpha_j \chi_{A_j}-\sum\limits_{j=1}^n q_j \chi_{D_j}\right\|^p d|y'\circ \nu|\\
& = \sup\limits_{\|y'\|\leq 1}\int_{\Omega}\sum\limits_{j=1}^n \|\alpha_j \chi_{A_j}- q_j \chi_{D_j}\|^p d|y'\circ \nu|\\
& =\sup\limits_{\|y'\|\leq 1}\sum\limits_{j=1}^n\int_{\Omega}\|\alpha_j \chi_{A_j}+\alpha_j\chi_{D_j}-\alpha_j\chi_{D_j}- q_j \chi_{D_j}\|^p d|y'\circ \nu|\\
& \leq \sup\limits_{\|y'\|\leq 1} \sum\limits_{j=1}^n \int_{\Omega}\|\alpha_j\|^p|\chi_{A_j}-\chi_{D_j}| d|y'\circ \nu|\\
& + \sup\limits_{\|y'\|\leq 1} \sum\limits_{j=1}^n \int_{\Omega}\|\alpha_j-q_j\|^p|\chi_{D_j}|d|y'\circ \nu|\\
& \leq \epsilon^p
\end{align*}

\begin{align*}
\left\|f-\sum\limits_{j=1}^n q_j\chi_{D_j}\right\|_{L^p_{w}(\nu,X,Y)} & = \|f-g\|_{L^p_{w}(\nu,X,Y)}+\left\|\sum\limits_{j=1}^n \alpha_j \chi_{A_j}-\sum\limits_{j=1}^n q_j \chi_{D_j}\right\|_{L^p_{w}(\nu,X,Y)}\\
& \leq 2\epsilon 
\end{align*}
This shows that the space $ L^p_{w}(
\nu,X,Y)$ is  separable.
\end{proof}
We now establish a H\"older inequality for our spaces.
\begin{prop}(H\"older inequality). \label{prop:Desigualdad de Holder}
	Let be $1< p < \infty$ and $\frac{1}{p}+\frac{1}{q}=1$. Let $(\Omega,\mathcal{A})$ be a mesurable space, $X,Y$ Banach spaces and  $\nu:\mathcal{A}\to Y$ a $\sigma$-additive vector measure. If $f\in L^p_{w}(\nu,X,Y)$ and $g\in L^q_{w}(\nu,X,Y)$. Then $$\sup\limits_{\|y'\|\leq 1}\left\{\int_{\Omega}\|f(w)\|_{X}\|g(w)\|_{X}d|y'\circ\nu|(w)\right\}\leq \|f\|_{L^p_{w}(\nu,X,Y)}\|g\|_{L^q_{w}(\nu,X,Y)}.$$
\end{prop}
\begin{proof}
	We will use Young's inequality $$ab\leq \frac{a^p}{p}+\frac{b^q}{q}, \quad a,b\geq 0, \quad \frac{1}{p}+\frac{1}{q}=1.$$
	In this case, let be $$a=\frac{\|f(w)\|}{\|f\|_{L^p_{w}(\nu,X,Y)}}\quad \quad b=\frac{\|g(w)\|}{\|g\|_{L^q_{w}(\nu,X,Y)}}.$$
	Then $$\sup\limits_{\|y'\|\leq 1}\left\{\int_{\Omega}\frac{\|f(w)\|\|g(w)\|}{\|f\|_{L^p_{w}(\nu,X,Y)}\|g\|_{L^q_{w}(\nu,X,Y)}}d|y'\circ\nu|(w)\right\} \leq \frac{1}{p}+\frac{1}{q}=1,$$
	so that
	$$\sup\limits_{\|y'\|\leq 1}\left\{\int_{\Omega}\|f(w)\|\|g(w)\|d|y'\circ\nu|(w)\right\}\leq \|f\|_{L^p_{w}(\nu,X,Y)}\|g\|_{L^q_{w}(\nu,X,Y)}. $$
\end{proof}

If $X$ is a Banach algebra, the H\"older inequality is modified as follows:

\begin{prop}(H\"older Inequality). \label{prop: Holder Inequalitymodified}
Let $1 < p <\infty$ and $\frac{1}{p}+\frac{1}{q} = 1 $. If $f \in L^p_{w}(\nu,X,Y) $ and $g\in L^q_{w}(\nu,X,Y)$
then $$\sup\limits_{\|y'\|\leq 1} \left \{\int _ {\Omega} \|f(w)g(w)\|_{X}d|y'\circ\nu|(w)\right\} 
\leq \|f \|_ {L^p_{w}(\nu,X,Y)} \| g \|_{L^q_{w}(\nu,X,Y)} .$$
	\end{prop}
In the above statement $f(w)g(w)$ denotes the multiplication of $f(w)$ and $g(w)$ in the algebra $X$.\\

%\begin{enumerate}
Despite the existence of reasonable H\"older inequalities, the study of duals for these spaces exhibit some pathologies. 
Thus, we devote the last part of this work to note some observations  regarding the duality of the spaces $L^ p_{w}(\nu, X, Y) $. 
	To do so, we will begin by stating the following definitions taken from \cite{KK}.
	\begin{defn}
	Let  $Y$ be a Banach space. We say that $Y$ has the 
	\textit{B-P property} if given a sequence  
	$(y_n)_{n \in \mathbb{N}}$ such that $\sum 
	\limits_{n=1}^\infty|\left\langle y', y_n 
	\right \rangle|<\infty $ for each $y' \in Y'$ there exists an element $ y \in Y $ with $ y = 
	\sum \limits_{n = 1 }^\infty y_n $.
	\end{defn}
	If $Y$ is a Banach space, then $Y$ has the  B-P property if and only if $Y$ does not contain an isometric copy of $ c_0 $.\\
	
	As a consequence of the above, we have the following theorem.
	\begin{thm}
	Let $Y$ be a Banach space with the B-P property  and $\nu:\mathcal{A} \to Y $ a $\sigma$-additive vector measure. If $f$ is a real valued measurable function  and 
	$y'\circ \nu $-integrable for each $y' \in Y'$ then $f$ is $\nu $-integrable.
	\end{thm}
	\begin{proof}
	See \cite{KK}, Theorem 1, Page 31.
	\end{proof}
	With this in mind, if $ X = \mathbb {R} $ and $Y$ has the B-P property, then  $L^p_{w}(\nu)= L^p(\nu) $ and thanks to the example by Sanchez in \cite{Sn} constructed in the case $Y=\ell^2$, one can show us that  these spaces have some obstacles to characterize their duals. In addition, if $Y$ has the B-P property, we can conclude that $L^p(\nu,X,Y)= L^p_w(\nu,X,Y)$ since the previous theorem is applied to real valued measurable function $\|f\|^p$.

\vspace{0.2cm}
	
\textbf{Acknowledments}

I would like to express my gratitude to Dr. Julio Delgado who suggested me to work on these $L^p$ spaces as part of my Ph.D. thesis and for the fruitful discussions during my visit to the Imperial College London.
%	\end{enumerate}
%Sanchez en ...utiliza este teorema para demostrar a traves de un ejemplo que existe una funcion de valor real que no pertenece a L^q_{w}(\nu X,Y) que seria el dual esperado de( L^p_{w}(\nu,X,Y))' encontrando  un x' para el cual la funcion  no es <x',\lambda>- integrable donde \lambda es una medida vectorial.

%\bibliographystyle{alphaabbr}

%\bibliography{bib-Delgado-13-10-29}

\begin{thebibliography}{Lew70}
		
			\bibitem[BD55]{BDS}
		R.G. Bartle, N. Dunford and J. Schwartz.
		\newblock{\em{W}eak compactness and vector measures.}
	    \newblock Canad. J. Math. 7 (1955), 289-305.
		
			\bibitem[OB16]{OB}
		O. Blasco.
		\newblock{\em{F}ourier analysis for vector-measures on compact abelian groups}.
		\newblock Racsam. 110:519--539, 2016.
		
		\bibitem[CS13]{CS}
		J.M. Calabuig, F. Galaz, E.M. Navarrete, E.A. S\'anchez.
		\newblock{\em{F}ourier Transform and Convolutions on $L^p$ of a Vector Measure on a Compact Hausdorff Abelian Group}.
		\newblock J. Fourier Anal Appl. 19: 312--332, 2013. 
		
	\bibitem[CB07]{CB1}
		N.~Chakraborty and S.~Basu.
		\newblock {\em{O}n some properties of the space of tensor integrable functions.}
		\newblock  Analysis Mathematica. 33:1--16, 2007.
		
		\bibitem[CB08]{CB2}
		N.~Chakraborty and S.~Basu.
		\newblock{\em{S}paces of p-tensor integrable functions and related Banach space
		properties.}
		\newblock Real Analysis Exchange. 34(1):87--103, 2008.
		
		\bibitem[DM09]{DM}
		O.~Delgado and P.~J. Miana.
		\newblock {\em{A}lgebra structure for $L^p$ of a vector measure.}
		\newblock Journal of Mathematical Analysis and Applications.
		(358):355--363, 2009.
		
		\bibitem[EO16]{EO}
		O. Edhan.
		\newblock{\em{V}alues of vector measure market games and their representations.}
		\newblock Internat. J. Game Theory. (45): 411--433, 2016.
		
		\bibitem[G.P94]{GCb}
		G.P.Curbera.
		\newblock {\em{W}hen $l^1$ of a vector measure is an AL-space.}
		\newblock  Pacific Journal of Mathematics. 162(2), 1994.
		
		\bibitem[GR75]{Cb}
		G.P.Curbera and W.~Rickerl.
		\newblock {\em{V}ector measures, Integration and Applications.}
		\newblock In: Positivity (in Trends Math.), Birkh\"auser, Basel, pp. 127--160 (2007).
		
		\bibitem[GL17]{GL}
		J. J. Grobler, C.C.A. Labuschagne.
		\newblock{\em {T}he Ito integral for martingales in vector lattices}.
		\newblock Journal of Mathematical Analysis and Applications.
		 (450): 1245--1274 (2017).
		
		%	\bibitem[DR16]{DR1}
		%J. Delgado y M. Ruzhansky.
		%\newblock{\em {A}pproximation property and nuclearity on mixed-norm $L^p$, modulation and Wiener amalgam spaces}.
		%\newblock J. London Math. Soc. (2) 94: 391--408, 2016.
		
		%	\bibitem[DR18]{DR}
		%J. Delgado y M. Ruzhansky.
		%\newblock{\em {T}he bounded approximation property of variable Lebesgue spaces and nuclarity}.
	  %\newblock	Math. Scan. 122: 299--319, 2018. 
		
		\bibitem[KG75]{KK}
		I.~Kluv\'anek and G.Knowles.
		\newblock {\em {V}ector measures and control systems}.
		\newblock North- Holland Amsterdam, 1975.
		
		\bibitem[Lew70]{L}
		D.~Lewis.
		\newblock {\em{I}ntegration with respect to vector measures}.
		\newblock J. Differential Equations. 33:157--165, 1970.
		
		\bibitem[P\'01]{Sn}
		E.~S. P\'erez.
		\newblock {\em{C}ompactness arguments for spaces of p-integrable functions with
		respect to a vector measure and factorization of operators through
		{L}ebesgue-{B}ochner spaces.}
		\newblock  Illinois Journal of Mathematics. 45(3):907--923, 2001.
		
		\bibitem[Ste93]{SG1}
		G.~F. Stef\'ansson.
		\newblock {$ l_1$ of a vector measure.}
		\newblock le Matematiche. 48(48):219--234, 1993.
		
		\bibitem[Ste11]{SG}
		G.~F. Stef\'ansson.
		\newblock {\em{I}ntegration in vector spaces.}
		\newblock Illinois Journal of Mathematics. 45(3):925--938, 2011.
		
	\end{thebibliography}
%\begin{thebibliography}{AGG03}

\end{document}